\documentclass{amsart}

\usepackage{verbatim}
\usepackage{amscd}
\usepackage{amsmath, amsfonts, amsthm}
\usepackage{gnuplot-lua-tikz}
\usepackage{hyperref}

\DeclareMathOperator{\Fix}{Fix}
\DeclareMathOperator{\MF}{MF}

\newcommand{\htp}{\simeq}

\newtheorem{thm}{Theorem}
\newtheorem{lem}[thm]{Lemma}

\newtheorem{conj}[thm]{Conjecture}

\theoremstyle{definition}
\newtheorem{exl}[thm]{Example}
\newtheorem{defn}[thm]{Definition}

\begin{document}

\bibliographystyle{hplain}

\title{The Wecken property for random maps on surfaces with boundary}
\author[J. Brimley]{Jacqueline Brimley}
\address{J. Brimley, Dept of Mathematics and Computer Science,
Fairfield University,
Fairfield, CT 06824}
\email{jacqueline.brimley@student.fairfield.edu}
\author[M. Griisser]{Matthew Griisser}
\address{M. Griisser, Department of Mathematics, Georgia Tech, Atlanta, GA 30332}
\email{mgriisser3@gatech.edu}
\author[A. Miller]{Allison Miller}
\address{A.~Miller, Department of Mathematics Pomona College Claremont, CA 91711}
\email{anm12008@mymail.pomona.edu}
\author[P. C. Staecker]{P. Christopher Staecker}
\address{P. C.~Staecker, Dept of Mathematics and Computer Science,
Fairfield University,
Fairfield, CT 06824}
\email{cstaecker@fairfield.edu}

\begin{abstract}
A selfmap is Wecken when the minimal number of fixed points among all maps in its homotopy class is equal to the Nielsen number, a homotopy invariant lower bound on the number of fixed points. All selfmaps are Wecken for manifolds of dimension not equal to 2, but some non-Wecken maps exist on surfaces.

We attempt to measure how common the Wecken property is on surfaces with boundary by estimating the proportion of maps which are Wecken, measured by asymptotic density. Intuitively, this is the probability that a randomly chosen homotopy class of maps consists of Wecken maps. We show that this density is nonzero for surfaces with boundary.

When the fundamental group of our space is free of rank $n$, we give nonzero lower bounds for the density of Wecken maps in terms of $n$, and compute the (nonzero) limit of these bounds as $n$ goes to infinity. 
\end{abstract}

\maketitle

When $X$ is a compact ANR and $f:X \to X$ is a selfmap, the \emph{Nielsen number} $N(f)$ is a homotopy and homotopy-type invariant which gives a lower bound for the cardinality of $\Fix(f)$, the fixed point set of $f$. We will focus on the case where $X$ has the homotopy-type of a surface with boundary, or equivalently a bouquet of circles. Several techniques have been developed for computing $N(f)$ in this setting, notably Wagner's algorithm of \cite{wagn99}, which succeeds for ``most'' maps $f$, in a technical sense which we will describe. 

Let $\MF(f)$ be the \emph{minimal number of fixed points}, defined as
\[ \MF(f) = \min \{\#\Fix(g) \mid g \htp f \}. \]
The Nielsen number is defined so that $N(f) \le \MF(f)$, and the classical work of Wecken \cite{weck41} shows that in fact these quantities are equal when $X$ is a manifold (with or without boundary) of dimension not equal to 2. The equality of $N(f)$ and $\MF(f)$ in dimension 2 was an open question for decades, until Jiang in \cite{jian84} gave an example of a map on the pants surface with $N(f) = 0$ but $\MF(f) = 2$.

We will say that a map $f$ is \emph{Wecken} when $N(f)=\MF(f)$. Several classes of Wecken maps have been identified in the literature. Jiang and Guo demonstrated that homeomorphisms are Wecken in \cite{jg93}. Wagner in \cite{wagn97} gave three classes of Wecken maps on the pants surface. 

Our goal in this paper is to measure the proportion of selfmaps on surfaces with boundary (and spaces of the same homotopy type) which are Wecken. We will measure this proportion in the language of asymptotic density and genericity (see \cite{ks08}).

For a free group $G$ and a natural number $p$, let $G_p$ be the subset
of all words of length at most $p$. The \emph{asymptotic density} (or simply
\emph{density}) of a subset $S\subset G$ is defined as 
\[ D(S) = \lim_{p \to \infty} \frac{|S \cap G_p|}{|G_p|}, \]
where $|\cdot|$ denotes the cardinality. The set $S$ is said to be
\emph{generic} if $D(S) = 1$.

Let $G= \langle a_1, \dots, a_n\rangle$ be the free group on $n$ generators. 
We will sometimes consider $\overline{G_p}$, the set of all words of length exactly $p$. When constructing a word in $\overline{G_p}$, there are $2n$ choices for the first letter and $2n-1$ choices for each of the $p-1$ subsequent letters, since we cannot choose the inverse of the immediately preceding letter. Thus $|\overline{G_p}| = 2n(2n-1)^{p-1}$, and so 
\[ \left|G_p \right| = 1+ \sum_{k=1}^{p}{\overline{G_k}}= 1+ \sum_{k=1}^{p}{(2n)(2n-1)^{k-1}}= \frac{n(2n-1)^p -1}{n-1}.
\]

Similarly, if $S \subset G^n$ is a set of $n$-tuples of
elements of $G$, the asymptotic density of $S$ is defined as
\[ D(S) = \lim_{p \to \infty} \frac{|S \cap G_p^n|}{|G_p^n|}, \]
and $S$ is called \emph{generic} if $D(S) = 1$.
Note that $|G_p^n|= |G_p|^n$. 

A homomorphism on free groups $G \to H$ with $G = \langle a_1,
\dots, a_n\rangle$ 
is equivalent combinatorially to an $n$-tuple of elements of $H$
(the $n$ elements are the words $\phi(a_1), \dots, \phi(a_n)$). Thus
the asymptotic density of a set of 
homomorphisms can be defined in the same sense as above, viewing the set of
homomorphisms as a collection of $n$-tuples. 

The homotopy class of a selfmap on a surface with boundary is determined by its induced map on the fundamental group, which is a free group. If $\phi:G \to G$ is the induced homomorphism of a Wecken map on the surface with fundamental group $G$, we say that $\phi$ is Wecken. Let $W_n$ be the set of Wecken homomorphisms of the free group on $n$ generators. We wish to estimate $D(W_n)$. This is the ``proportion of selfmaps which are Wecken''  discussed informally above. 

Existing classifications of Wecken maps are not very informative concerning the question of the asymptotic density. Jiang and Guo's result in \cite{jg93} that homeomorphisms are Wecken gives no information, since the density of homomorphisms $G \to G$ which are isomorphisms is zero (a result in \cite{stae11} shows that the set of surjections has density 0). Wagner's work in \cite{wagn97} only addresses maps for the case where $n=2$ (the pants surface). This will lead to nonzero lower bounds on $D(W_2)$, but it will not allow any measurement of $D(W_n)$ for $n \neq 2$.

The work of this paper was undertaken with a somewhat experimental approach. Before our detailed work began, randomized computer tests were used to attempt some informal estimations of the quantities to be computed in this paper. The whole paper can be seen as an attempt to rigorously prove various properties which became immediately clear upon examining these computer simulations.

Our main results will consist of various bounds on $D(W_n)$ and a related quantity $D(V_n) \le D(W_n)$ for a certain related set $V_n$. In particular we show that $D(W_n) \neq 0$ for all $n$, and that $\lim_{n\to\infty} D(W_n) \neq 0$. These are new results, and we additionally give nonzero lower bounds for these quantities. Our actual derived lower bounds are somewhat complicated, and we summarize 
some of their values 
in Table \ref{boundstable}. Since $D(V_n) \le D(W_n)$, the lower bounds on $D(V_n)$ are also lower bounds on $D(W_n)$. The upper bounds on $D(V_n)$ are not necessarily upper bounds on $D(W_n)$, but should be viewed as theoretical upper bounds on the effectiveness of our approach in estimating $D(W_n)$. The particular values of $n$ given in the table were chosen arbitrarily-- we require a special argument for $n=2$, but all other values of $n$ are treated by the same methods.

The $n=2$ case is handled by Theorem \ref{n=2bound}, the upper bounds on $D(V_n)$ for $n>2$ are from Theorem \ref{upperbd}, and the lower bounds are from Theorem \ref{bestbound}. We also consider the quantity $\lim_{n\to\infty} D(W_n)$. We prove that this quantity is at least $e^{-3}\approx .0497$ and give strong evidence that it is at least $e^{-2} \approx .1353$. Based on our computer experiments, we conjecture that it is at least $e^{-1} \approx .3678$.

\begin{table}
\begin{center}
\begin{tabular}{c|r@{ $\ge$ }c@{ $\ge$ }l}
$n$ & \multicolumn{3}{c}{Density bounds} \\
\hline
2 & \multicolumn{3}{c}{$D(W_2) \ge .2129$} \\
3 & .3403 & $D(V_3)$ & .0252 \\
5 & .3430 & $D(V_5)$ & .0694\\
10 & .3590 & $D(V_{10})$ & .1029 \\
20 & .3634 & $D(V_{20})$ & .1193 \\
50 & .3661 & $D(V_{50})$ & .1289 
\end{tabular}
\caption{Best derived bounds for $D(V_n)$ and $D(W_n)$ for various $n$\label{boundstable}}
\end{center}
\end{table}

This paper contains some computations and algebraic simplifications which require a computer. Where possible we have used the open source computer algebra system Sage, but some computations have required Mathematica. Sage and Mathematica code for all computations can be found at the last author's website.\footnote{\url{http://faculty.fairfield.edu/cstaecker}}

The structure of the paper is as follows: In Section \ref{wagner} we discuss the details of Wagner's algorithm, and present our experimental results. In the following section we give a lower bound for the density of certain classes of Wecken maps that Wagner identifies for $n=2$. In Section \ref{der} we give a relatively simple proof that $D(W_n) > 0$ for $n > 2$, and in Section \ref{new} we give a more sophisticated lower bound for $D(W_n)$ and discuss the limit of this bound as $n$ goes to infinity.

We would like to thank the Fairfield University Sciences Institute and the organizers of the International Conference on Nielsen Fixed Point Theory and Related Topics at Capital Normal University in Beijing, China for graciously supporting our attendance at the conference. This paper is the product of a summer REU project at Fairfield University supported by the National Science Foundation and the Department of Defense under Grant No.~1004346.

\section{Wagner's algorithm and Wecken maps}\label{wagner}
Our basic approach to estimating the density of Wecken maps is via Wagner's algorithm of \cite{wagn99} for computation of the Nielsen number. We begin with a brief overview of the technique.

First, any self-map $f$ on a bouquet of circles that induces the homomorphism $\phi: G \to G$ is homotopic to a ``standard form" map (described in \cite{wagn99}). A map in Wagner's standard form has a fixed point $x_0$ at the wedge point and a fixed point identified with each occurrence of $a_i$ or $a_i^{-1}$ in the reduced word form of $\phi(a_i)$. 

Wagner's algorithm only works with maps which obey the following \emph{remnant} condition:  An endomorphism $\phi:G \to G$ is said to \emph{have remnant} when, for every $i$, there is a nontrivial subword of $\phi(a_i)$ that does not cancel in all products of the form 
\[ \phi(a_j)^{\pm1}\phi(a_i)\phi(a_k)^{\pm1} \]
except for when $j$ or $k$ equals $i$ and the exponent is $-1$.  

A theorem of Robert F. Brown in \cite{wagn99} established that the set of endomorphisms with remnant is generic.

For a given endomorphism $\phi:G \to G$ we define the set of \emph{Wagner tails}, which are elements of $G$. For each occurence of the letter $a^\epsilon_i$ (where $\epsilon$ $\in$ \{$+1$,$-1$\}) in $\phi(a_i)$, we write a reduced product $\phi(a_i)$ = $va^{\epsilon}_i\overline{v}$. Then $w$ and $\overline{w}$ are Wagner tails, where:
\begin{displaymath}
w =
\begin{cases}       
v & $if $ \epsilon = 1\\       
va^{-1}_i & $if $ \epsilon = -1     
\end{cases}   
\qquad
\overline{w} =
\begin{cases}
\overline{v}^{-1} & $if $ \epsilon = 1\\
\overline{v}^{-1}a_i & $if $ \epsilon = -1
\end{cases}
\end{displaymath}
In addition, for the wedge point we say that $w = \overline{w} = 1$.

We say that such Wagner tails $w$ and $\overline{w}$ \emph{arise from an occurence of} $a^{\epsilon}_i$ in $\phi(a_i)$. Wagner shows that the fixed point index of $x_{i}$ is equal to $-\epsilon$ and that the fixed point index of $x_0$ is 1.
Letting $w_1,\overline{w}_1$ and $w_2,\overline{w}_2$ be Wagner tails arising from two different fixed points, the fixed points are \emph{directly related} when $\{w_1,\overline{w}_1\} \cap \{w_2,\overline{w}_2\} \neq \emptyset$. Two fixed points $x, y$ are  \emph{indirectly related} when there is a sequence of fixed points $x=x_0, x_{1}, \ldots, x_{k-1}, x_{k}=y$ such that $x_{i}$ is directly related to $x_{i+1}$ for $0 \leq i < k$. This indirect relation is an equivalence relation. 

Theorem 3.7 of \cite{wagn99} shows that when $\phi$ has remnant, the number of such equivalence classes with nonzero fixed point index sum (these are called \emph{essential fixed point classes}) is equal to $N(f)$.

\begin{exl}\label{ex}
Let $f: X \to X$ be a self map and $\phi:G \to G$ be its induced homomorphism, where $G = \langle a,b \rangle$, $\phi(a) = ba^3b$, and $\phi(b) = ab^{-1}a^2$.

The three occurrences of the letter $a$ in $\phi(a)$ indicate that there are three fixed points inside the loop $a$ (which we will call $x_1, x_2, $ and $x_3$) and the single occurrence of $b^{-1}$ in $\phi(b)$ indicates that there is one fixed point inside the loop $b$ (which we will call $x_4$). The wedge point is also a fixed point, and we will denote it as $x_0$. Then we can compute the Wagner tails and indices as in  Table \ref{tails}.
\begin{table}[htdp]
\begin{center}
\begin{tabular}{| c c c c}
Fixed Point & Index & $w$ & $\overline{w}$ \\
\hline
$x_0$ & $+1$ & $1$ & $1$ \\
$x_1$ & $-1$ & $b$ & $b^{-1}a^{-2}$ \\
$x_2$ & $-1$ & $ba$ & $b^{-1}a^{-1}$ \\
$x_3$ & $-1$ & $ba^{2}$ & $b^{-1}$ \\
$x_4$ & $+1$ & $ab^{-1}$ & $a^{-2}b$ \\
\end{tabular}
\end{center}
\caption{Wagner tails for Example \ref{ex} \label{tails}}
\end{table}%

As none of the Wagner tails are directly related, each fixed point belongs to its own class. This means that each class has a nonzero fixed point index sum and that each class is consequently essential. It is a classical result of Nielsen theory that fixed points from different classes can never be combined by homotopy. Therefore, $N(\phi)=\MF(\phi)$ = 5, and our example is Wecken.
\end{exl}

Our fundamental tool for identifying a Wecken map is the following easy observation:

\begin{lem}\label{differenttails}
Let $\phi:G \to G$ be a homomorphism with remnant. If the Wagner tails of $\phi$ are all different (except for the repeated word 1 at the wedge point), then $\phi$ is Wecken.
\end{lem}

\begin{proof}
If every Wagner tail of $\phi$ is different (except for the repeated word 1 at the wedge point), then every fixed point of $f$ (a continuous map whose induced homomorphism is $\phi$) is in its own fixed point class. Because each fixed point has index $\pm 1$ and is the only element in its fixed point class, the sum of the indices for each fixed point class is $\pm 1$. So, every fixed point class is essential. This implies $N(\phi) = \MF(\phi)$, and hence $f$ is Wecken.
\end{proof}

When $G$ is the free group on $n$ generators, let $V_n$ be the set of endomorphisms of $G$ whose Wagner tails are all different, let $W_n$ be the Wecken endomorphisms, and let $R_n$ be the endomorphisms with remnant. The above Lemma is that $V_n\cap R_n \subset W_n$, and thus $D(V_n\cap R_n) \le D(W_n)$. Since $R_n$ is generic we have $D(V_n\cap R_n) = D(V_n)$, and so $D(V_n) \le D(W_n)$.

The density of $V_n$ can be measured experimentally by generating homomorphisms of $G$ at random using particular bounded word lengths, and testing if they are in $V_n$. This test is easily done on a computer. At the outset of this project we computed several of these random trials, resulting in the data presented in Table \ref{data}.

\begin{table}
\begin{center}
\input{plot.tex}

\vspace{1cm}

\begin{tabular}{|c|l l l l l l l|}
\hline
 & \multicolumn{7}{c|}{$D_p(V_n)$ approximations} \\
$p$ & $n=2$ & $n=3$ & $n=4$ & $n=5$ & $n=10$ & $n=20$ & $n=50$ \\
\hline
2	& .3916	& .3705	& .3721	& .3763	& .3732	& .3637	& .3683	\\ 
3	& .3128	& .2980	& .3068	& .3210	& .3338	& .3472	& .3597	\\
4	& .2474	& .2628	& .2935	& .3075	& .3400	& .3496	& .3632	\\
5	& .2227	& .2534	& .2775	& .2956	& .3354	& .3572	& .3676	\\
6	& .1987	& .2476	& .2727	& .2976	& .3326	& .3549	& .3631	\\
7	& .1969	& .2414	& .2776	& .2980	& .3206	& .3356	& .3621	\\
8	& .1949	& .2506	& .2746	& .2896	& .3336	& .3458	& .3577	\\
9	& .1887	& .2522	& .2737	& .3002	& .3302	& .3496	& .3582	\\
10	& .1885	& .2542	& .2732	& .2982	& .3400	& .3427	& .3615	\\
11	& .1902	& .2412	& .2755	& .2961	& .3341	& .3465	& .3604	\\
12	& .1920	& .2376	& .2727	& .2983	& .3367	& .3446	& .3617	\\
13	& .1926	& .2512	& .2732	& .2970	& .3315	& .3540	& .3608	\\
14	& .1934	& .2482	& .2822	& .2931	& .3323	& .3513	& .3488 \\
\hline
\end{tabular}
\end{center}
\caption{Experimental data for various $n$\label{data}}
\end{table}

The values approximated in the table are $D_p(V_n) = \frac{|V_n\cap G_p^n|}{|G_p^n|}$ for $p \in \{2, \dots, 14\}$ and $n\in\{2,3,4,5,10,20,50\}$. These values of $n$ were chosen arbitrarily-- the behavior seems to be similar for any $n$. These approximations were computed by producing 10000 random homomorphisms of $G_p^n$ and measuring exactly the proportion of these which are in $V_n$. 

The data immediately suggest that $D(V_n)$, and thus $D(W_n)$, is nonzero for all $n\ge 2$. They also suggest that $D(V_n)$ is increasing in $n$ and has a limit in $n$ which is less than 1. The dotted lines on the chart in Table \ref{data} indicate some values which will appear in our results in the following sections. 

\section{Wagner's classes of Wecken homomorphisms for the case $n=2$}

Wagner identifies three classes of Wecken maps on the pants surface (where $n=2$) which she calls $T_2$, $T_4$, and $T_5$. We will discuss $T_2$, which we will subdivide into $T_{2a}$ and $T_{2b}$, and $T_4$. The class $T_5$ is not of interest to us since it has density zero (all maps in $T_5$ do not have remnant). 
The classes $T_2$ and $T_4$ consist of maps whose induced homomorphisms have remnant and are ``simple'', where simplicity is defined as follows.

\begin{defn}
For two words $x$ and $y$ with $x \neq y$, let $M(x,y)$ be the possibly trivial maximal initial subword of $x$ that cancels in $y^{-1}x$. Equivalently, $M(x, y)$ is the maximal initial subword that $x$ and $y$ share. 

A set $S$ is \emph{simple} if there exists a word $U$ such that for all $x, y$ in $S$ with $x \neq y$ we have either $M(x, y)=1$ or $M(x, y)=U$. A homomorphism $\phi: \langle a_1, \dots, a_n \rangle \to \langle a_1, \dots, a_n \rangle$ is \emph{simple} if the set $S_{\phi}= \{ \phi(a_1), \phi(a_1)^{-1}, \dots, \phi(a_n), \phi(a_n)^{-1} \}$ is simple. 
\end{defn}

For the rest of this section we will focus on the case where $n=2$, and we write our group as $G= \langle a, b \rangle$. 

Wagner's class $T_2$ can be split into two subclasses. We say that a homomorphism $\phi$ is $T_{2a}$ if $M(x,y)$ is trivial for all $x,y \in S_\phi$.
Let $s_a$  equal the first letter (generator or inverse of a generator) of $\phi(a)$ and $l_a$ be the inverse of the last letter of $\phi(a)$, and similarly define
$s_b$ and $l_b$. 
Then Type $T_{2a}$ occurs exactly when $s_a, l_a, s_b, l_b$ are four distinct letters:
since our group has only four distinct letters, 
we must have $\{ s_a, l_a, s_b, l_b \}= \{a, a^{-1}, b, b^{-1} \}$. 

\begin{lem} $D(T_{2a}) \geq \frac{2}{27}$ \end{lem}

\begin{proof}
We will count the number of ways to construct a map $\phi \in G_p^2 \cap T_{2a}$.

First determine the assignment of $\{ s_a, l_a, s_b, l_b\}$ to $ \{a, a^{-1}, b, b^{-1} \}$, in one of $4!=24$ ways. 
Then we have to determine the inner letters of $\phi(a)$ and $\phi(b)$ in one of at least
\begin{align*}
\left( \sum_{q=3}^{p} (3^{q-3})2 \right) ^2 = \left( 3^{p-2}-1\right)^2 
\end{align*}
ways: each of the inner letters can be anything except for the inverse of the previous (thus 3 choices for each), and the last letter has possibly only 2 choices since it additionally cannot be the inverse of the already chosen final letter.

So we have 
\[
 D(T_{2a}) \geq \lim_{p \to \infty} 24 \frac{\left( 3^{p-2}-1\right)^2}{ \left(\frac{2(3^{p})-1}{1}\right)^2}
= \lim_{p \to \infty} 24 \frac{3^{2p-4}}{4(3^{2p})}= \frac{24}{4(3^4)}= \frac{2}{27}
\qedhere
\]
\end{proof} 

Wagner's class $T_2$ also includes a subclass which we call $T_{2b}$. A homomorphism $\phi$ is $T_{2b}$ when we can write $\phi(a)= U X U^{-1}$ for some nontrivial $U$ and $\phi(b)$ is nontrivial such that $s_b \neq l_b$ and $s_a \not\in \{s_b, l_b\}$.
 
\begin{lem}
$
D(T_{2b}) \geq \frac{1}{24}
$
\end{lem}
 
\begin{proof}
For $\phi$ in $T_{2b} \cap G_p^2$, let  $v= |U|$, and let $r= |X| = |\phi(a)|-2v$. (We assume that $U$ and $X$ are chosen so that $U$ is maximal when we write $\phi(a) = UXU^{-1}$, which is equivalent to assuming that the last letter of $X$ is not the inverse of the first letter.) We will count the number of $\phi$ in $T_{2b} \cap G_p^2$ with a particular $v$ and $r$.  
There are $4(3^{v-1})$ ways to select $U$, which can be any word of length $v$. 

For each of the first $r-1$ letters of $X$, we can choose any letter but the inverse of the immediately preceding letter. For the last letter of $X$, we cannot choose the inverse of the immediately preceding letter, the inverse of the immediately following letter, or the inverse of the first letter of $X$. There will be perhaps only one choice remaining. Thus there are at least $3^{r-1}$ ways to determine $X$.

Summing over $v$ and $r$ there are 
\begin{align*}
\sum_{v=1}^{\frac{p-2}{2}} 4(3^{v-1})& \sum_{r=1}^{p-2v} 3^{r-1}
= \frac{ 3^p -4(3^{p/2}-1)-1}{3}
\end{align*}
 ways to determine $\phi(a)$. 
 Then we must determine $\phi(b)$, in one of 
 \[
 \sum_{q=2}^{p} 3(3^{q-2}) =\frac{3}{2} \left(3^{p-1}-1 \right) 
 \]
 ways, since for the first letter we cannot choose $s_a$ and for the last letter we cannot choose $s_a^{-1}$, $s_b^{-1}$ or the inverse of the immediately preceding letter. 
 So
\[
 D(T_{2b}) \geq
 \lim_{p \to \infty}
\frac{ \left( \frac{ 3^p -4(3^{p/2}-1)-1}{3} \right) \left( \frac{3}{2} \left(3^{p-1}-1 \right)  \right) }
{ \left( \frac{2(3)^p-1}{1} \right)^2}
= \lim_{p \to \infty} \frac{3^{2p-1}}{8(3^{2p})} = \frac{1}{24}
\qedhere
\]
 \end{proof}
 
There is an analogous class $T_{2b'}$, the set of $\phi$ with $\phi(b)= U X U^{-1}$ (so $s_b=l_b$) and $\phi(a)$ is nontrivial such that $s_a \neq l_a$ and $s_a, l_a \neq s_b$.

We can get the identical lower bound on this class: $D(T_{2b'}) \ge \frac{1}{24}$.

Wagner also identifies the set of $T_{4}$ homomorphisms, where $\phi(a)= UX_1$ and $\phi(b)= X_2 U^{-1}$ and we have $s_a \neq l_a$, $s_b \neq l_b$ and $l_a \neq s_b$. 
\begin{lem}
$ D(T_{4}) \geq \frac{1}{36} $
\end{lem}

\begin{proof}

Let $v= |U|$, $r= |X_1|$, $s= |X_2|$. (Again we assume that $U$ is chosen maximally.)
Then we have $4(3^{v-1})$ ways to determine $U$. There are at least $\sum_{r=1}^{p-v} 2(3^{r-1})$ ways to determine $X_1$, since we must make $l_a \neq s_a$. 
There are at least $ \sum_{s=2}^{p-v}2 (3^{s-2}) (2) $ ways to determine $X_2$, since we must have $s_b \not \in \{s_a, l_b^{-1}\}$ and additionally require that the last letter of $X_2$ is not the inverse of the first letter of $X_1$, so that $U$ will be maximal. (We have assumed for simplicity that $|X_2| = s > 1$, this has no effect in the limit as $p\to \infty$.) 

We have
\begin{align*}
\sum_{v=1} ^{p-1} \left( 4(3^{v-1})
 \left( \sum_{r=1}^{p-v}2( 3^{r-1}) \right) \left( \sum_{s=2}^{p-v} 2(3^{s-2}) \right) \right) 
&= \sum_{v=1} ^{p-1} 4 (3^{v-1}) (3^{p-v} -1) (3^{p-v-1}-1)\\
&= 3^{2p-2} - 4p3^{p-2} + 16(3^{p-2}) -1 
\end{align*}
ways to determine $\phi$ in $T_{4} \cap G_p^2$, and so
\[ D(T_{4}) \geq \lim_{p \to \infty} \frac{3^{2p-2} - 4p3^{p-2} + 16(3^{p-2}) -1}{\left( \frac{2(3^p)-1}{1} \right)^2}
= \lim_{ p \to \infty} \frac{ 3^{2p-2}}{ 4(3^{2p})} = \frac{1}{36}
\qedhere \]
\end{proof}

As above, we can switch the roles of $a$ and $b$ to get a class which we call $T_{4'}$ homomorphisms, and we have $D(T_{4'}) \ge \frac{1}{18}$.

Theorem 3.2 of \cite{wagn97} shows that all maps in $T_2$ and $T_4$ are Wecken. Identical proofs show additionally that maps in $T_{2b'}$ and $T_{4'}$ are Wecken. 

Since $T_{2a}, T_{2b},T_{2b'}, T_{4},$ and $T_{4'}$ are mutually disjoint and Wecken, we have 
\[ D(W_2) \geq D(T_{2a}) + D(T_{2b}) + D(T_{2b'}) + D(T_4) + D(T_{4'}), \]
and summing the lower bounds in the lemmas above gives:
\begin{thm}\label{n=2bound}
\[D(W_2) \geq \frac{23}{108} \approx .2129 \]
\end{thm}

Note that this is larger than the experimental data we have for $D(V_2)$ (see Table \ref{data}). This suggests that $W_2$ is strictly larger than $V_2$, and in fact this is not hard to demonstrate:

\begin{thm}
$D(W_2) > D(V_2)$
\end{thm}
\begin{proof}
Consider the set of $\phi$ such that $s_a= a$ and $\{ l_a^{-1}, s_b, l_b^{-1} \}= \{a^{-1}, b, b^{-1} \}$. 
This set is disjoint from $V_2$, but is a subset of $T_{2a} \subset W_2$. It also has a nonzero density of $\frac{1}{81}$. 
 So $D(W_2) > D(V_2)$. 
\end{proof}
 
Wagner's proof that the classes $T_2$ and $T_4$ are Wecken involves computing $N(f)$ with Wagner's algorithm,  computing $\MF(f)$ with an algorithm of Kelly in \cite{kell87}, and observing that $N(f)=\MF(f)$ for $f$ such that $\phi \in T_2 \cup T_4$. Kelly's algorithm for $\MF(f)$ is specific to the $n=2$ setting and has not been extended to $n>2$. Thus we will require totally different methods for general $n$. 

\section{A proof that $D(W_n) > 0$ for general $n$}\label{der}

In this section we show that $D(W_n)>0$ by showing that $D(V_n)>0$ for $n>2$.
Let $G = \left\langle a_1, \ldots, a_n \right\rangle$ and $\phi: G \to G$ be an endomorphism. If $x$ is a fixed point, we denote the corresponding Wagner tails $w$ and $\overline{w}$ by $w_{x}$ and $\overline{w}_{x}$. We also denote the first letter of $\phi(a_i)$ as $s_{i}$ and the inverse of the last letter of $\phi(a_i)$ as $l_{i}$. In general, we will denote the length of the reduced form of a word $w$ as $\left| w \right|$. 

\begin{lem}
\label{dirrel}
If $T_{\phi} = \left\{s_{1}, s_{2}, \ldots, s_{n}, l_{1}, l_{2}, \ldots, l_{n}\right\}$ contains no repeated elements and $a_{i} \notin \left\{s_{i}, l_{i}^{-1}\right\}$ for all $i$, then there are no direct relations among the Wagner tails of $\phi$.
\end{lem}

\begin{proof}
Let $x$ be a fixed point arising from an occurrence $a_i^{\pm1}$ in  $\phi(a_i)$. Then we can write $\phi(a_i)=v_{x}x\overline{v}_x$, where $x=a_i^{\pm1}$. 
If $v_x$ is empty, then we must have $s_i= a_i^{-1}=x$. Then $w_x=a_i^{-1}$ is non-trivial.
Similarly, if $\overline{v_x}$ is empty, we have $\overline{w_x}=a_i.$
For $v_x$ or $\overline{v}_x$ non trivial,  $w_x$ and $\overline{w}_x$ must also be not trivial.
So  $x$ is not directly related to the base point. 

To show that $x$ is not directly related to any other fixed point, let $y \neq x$ be a fixed point arising from an occurrence of  $a_j^{\pm1}$ in $\phi(a_{j})$. 

If $i = j$, assume without loss of generality that $x$ comes before $y$ in the word $\phi(a_{i})$. Then $\left| w_{x} \right| < \left| w_{y} \right|$ and $\left| \overline{w}_{x} \right| > \left| \overline{w}_{y} \right|$.  Since $w_x$ begins with $s_i$ and $\overline{w}_y$ begins with $l_i \neq s_i$, $w_x \neq \overline{w}_y$ and, by a similar argument $\overline{w}_x \neq w_y$. Hence, the Wagner tails of $x$ and $y$ are distinct, and so $x$ and $y$ are not directly related.

If $i \neq j$, we have that the first letters of $w_{x}$, $w_{y}$, $\overline{w}_{x}$, and $\overline{w}_{y}$ are $s_{i}$, $s_{j}$, $l_{i}$, and $l_{j}$, respectively. By construction, all four letters are distinct, and so $\{w_x, \overline{w}_x\} \cap \{w_y, \overline{w}_y\} = \emptyset$. Therefore $x$ and $y$ are not directly related.
\end{proof}

Our lower bound for $D(V_n)$ involves the number of derangements on $n$ elements. Recall that a derangement is a fixed point free rearrangement of a set, and that the number of derangements on a set of $n$ elements is given by the formula $!n = n! \sum_{i=0}^{n} \frac{(-1)^i}{i!}$. Asymptotically we have $!n\approx \frac{n!}{e}$ for large $n$.

\begin{thm}\label{derthm}

$
D(W_n) \geq D(V_n) > 0.
$
More precisely,
$$
D(V_n) \geq c_n = 2^n \left(!n\right)^{2} \left(\frac{n - 1}{n \left(2n - 1\right)^{2}}\right)^{n},
$$
where $!n$ is the number of derangements on $n$ elements.  
\end{thm}

\begin{proof}
For a given homomorphism  $ \phi: G \to G$, let $S_{\phi} = \{ s_1, s_2, \dots s_n \}$ and $L_{\phi}= \{ l_1, l_2, \dots l_n \}$.
Let $A_n$ be the set of all $\phi$ such that
$S_{\phi}=L_{\phi}=  \{ a_1^{\epsilon_1}, a_2^{\epsilon_2}, \dots a_n^{\epsilon_n}\}$ where $\epsilon_i = \pm 1$ and $\{s_i, l_i\} \cap \{a_i, a_i^{-1}\} = \emptyset.$
Then, by Lemma \ref{dirrel}, no $ \phi $ in $A_n$ has direct relationships among its Wagner tails, and so $A_n \subseteq V_n$ and $D(A_n) \leq D(V_n)$. 

To find the density of $A_n$, we must first find a lower bound for $|G_p^n \cap A_n|$. We will count the number of choices in a construction of $\phi$ in $A_n$.

First choose the set $\{ \epsilon_1, \epsilon_2, \dots, \epsilon_n\}$. This corresponds to choosing, for each $i=1, 2, \dots ,n$, whether $a_i$ or  $a_i^{-1}$ will appear in $S_{\phi}$. There are $2^n$ ways to make this choice. 

Then assign the set $\{ a_1^{\epsilon_1}, a_2^{\epsilon_2}, \dots a_n^{\epsilon_n}\}$ to $\{s_1, s_2, \dots, s_n\}$ such that $s_i \neq a_i^{\epsilon_i}$. Each of these assignments is a derangement on $\{1,2,\dots, n\}$, and so there are $!n$ ways. 
We similarly assign $\{ a_1^{\epsilon_1}, a_2^{\epsilon_2}, \dots a_n^{\epsilon_n}\}$ to $\{l_1,l_2, \dots, l_n\}$ such that $l_i \neq a_i^{\epsilon_i}$ in $!n$ ways. 

For each $i$, let $p_i=|\phi(a_i)|$.
 If $p_i=2$, then since we have determined the first ($s_i$) and last ($l_i$) letters of $\phi(a_i)$, we have no further choices. For $p_i > 3$, we have at least $(2n-1)^{p_i - 3}(2n-2)$ ways of filling the interior: for the second through $(p_i - 2)^{\text{th}}$ letters, there are $2n-1$ choices (any letter but the inverse of the one immediately before), and for the $(p_i - 1)^{\text{th}}$ letter, we have at least $2n-2$ choices (in the worst case, we can choose neither the inverse of the previous letter nor the inverse of the last letter).
 
 So, assuming $p\geq 2$ and  letting $p_i$ range over $2, 3, \dots,p$, we have at least
 
$$
1 + \sum_{p_i=3}^{p} (2n-2)(2n-1)^{p_i - 3} = (2n-1)^{p-2}
$$
 ways to fill in the interior of $\phi(a_i)$. 

Therefore, for $p \geq 2$ we have
$$
|G_p^n \cap A_n| \geq 2^n(!n)^2((2n-1)^{p-2})^n,
$$
and
\begin{align*}
D(A_n)&= \lim_{p\to\infty} \frac{|G_p^n \cap A_n|}{|G_p^n|} \geq \lim_{p\to\infty} \frac{2^n(!n)^2((2n-1)^{p-2})^n}{(\frac{n(2n-1)^p-1}{n-1})^n} \\
&=2^n  (!n)^2\lim_{p\to\infty} \left(\frac{(2n-1)^{p-2}(n-1)}{n(2n-1)^{p} -1}\right)^n =2^n (!n)^2 \left(\frac{n-1}{n(2n-1)^2}\right)^n 
\end{align*}
\end{proof}

The above lower bound $c_n$ gives very small values which decrease in $n$. They are approximated as:
\[ \begin{tabular}{c | c c c c c c c c }
$n$ & 2 & 3 & 5 & 10 & 20 & 50 \\
\hline
$c_n \text{ (approx.)} $ & $10^{-2}$ & $10^{-4}$ & $10^{-6}$ & $10^{-11}$ & $10^{-23}$ & $10^{-57}$
\end{tabular} \]
These values are very small, especially in light of the experimental data which suggests much greater values for $D(V_n)$. Further, it is easy to see that $c_n \to 0$ as $n\to\infty$, which does not match the experimental data. 

These lower bounds can be improved if we allow $\phi(a_i)$ to begin or end with $a_i^{-1}$, which was not allowed in the above construction. This results in a slightly better lower bound for $D(V_n)$, but one that still does not match the experimental values and has limit 0 as $n\to\infty$. To achieve any substantial improvement in the lower bounds, we must allow the letters $s_i$ and $l_i$ to become words of arbitrary length. This we do in the next section.

An argument similar to that in Theorem \ref{derthm} does allow us, however, to compute an (apparently) asymptotically sharp upper bound for $D(V_n)$.

\begin{thm}\label{upperbd}
Let $V_n$ be the set of Wecken endomorphisms with no direct relations among their Wagner tails. Then 
$$
D(V_n) \leq \left (1-\frac{1}{n} + \frac{1}{2n(2n-1)}\right)^{n}
$$
\end{thm}

\begin{proof}
Let $B_n$ be the set of homomorphisms $\phi$ such that for some $i$ we have $\phi(a_i)$ beginning with $a_i$, and $C_n$ be the set of $\phi$ such that for some $i$ we have $\phi(a_i)$ ending in $a_i$.  Then for all $\phi$ in $B_n \cup C_n$,  one of the Wagner tails corresponding to this occurrence of $a_i$ in $\phi(a_i)$ will be trivial and so it will equal the tail of the wedge point. 

So $B_n \cup C_n$ and $V_n$ are disjoint sets. 

Let $B(p)$ equal the set of words on $n$ generators of length $\leq p$ that begin with $a_i$ and $C(p)$ equal the set of words on $n$ generators of length $\leq p$ that end with $a_i$. 
Then we have

\begin{align*}
| V_n \cap G_p^n | & \leq \left(|G_p|- |B(p) \cup C(p)| \right)^n = \left( \sum_{q=1}^{p} 2n(2n-1)^{q-1} - |B(p) \cup C(p)| \right)^n \\
&=\left( \sum_{q=1}^{p} 2n(2n-1)^{q-1} - |B(p)|- |C(p)| + | B(p) \cap C(p) | \right) ^n\\
&\leq \left( \sum_{q=1}^{p} 2n(2n-1)^{q-1} - 2 \sum_{q=1}^{p}(2n-1)^{q-1} + \sum_{q=2}^{p} (2n-1)^{q-2}\right)^n \\
&= \left( \frac{n(2n-1)^p-1}{n-1} -  2 \frac{(2n-1)^{p}-1}{2n-2} + \frac{(2n-1)^{p-1} -1}{2n-2} \right) ^n.
\end{align*}

So 
\begin{align*}
D(V_n) &\leq \lim_{p \to \infty} \frac{  \left( \frac{(n-1)(2n-1)^{p} - 2 + .5 (2n-1)^{p-1} -.5}{n-1} \right)^n } 
{\left( \frac{ n(2n-1)^p-1}{n-1} \right)^n } \\
&= \lim_{p \to \infty} \left( \frac{ (n-1)(2n-1)^{p}-2 + .5(2n-1)^{p-1} - .5}{ (2n-2)( n(2n-1)^p-1)} \right) ^n \\
&= \lim_{p \to \infty} \left( \frac{(n-1)(2n-1)^p}{n(2n-1)^p} + \frac{(2n-1)^{p-1}}{2n(2n-1)^{p}}   \right)^n \\
&= \left( \frac{n-1}{n} + \frac{1}{2n(2n-1)}   \right)^n= \left( 1 - \frac{1}{n} + \frac{1}{2n(2n-1)} \right)^n.
\end{align*}

\end{proof}

This bound takes the following values:

\[ \begin{tabular}{c |  c c c c c c  }
$n$ &  2 & 3 & 5 & 10 & 20 & 50\\
\hline
$ \left( 1 - \frac{1}{n} + \frac{1}{2n(2n-1)} \right)^n$ & .3403 & .3430 & .3511 & .3590 & .3634 & .3661
\end{tabular} \]

Note that these upper bounds on $D(V_n)$ need not be upper bounds on $D(W_n)$. Indeed we have already seen that $D(W_2) > D(V_2)$. Thus the upper bound on $D(V_n)$ should be seen not as an upper bound on the density of Wecken homomorphisms, but as an upper bound on the effectiveness of Lemma \ref{differenttails} for approaching this problem.

One of questions suggested by the data concerns the computation of the limit $\lim_{n \to \infty} D(V_n)$. In order to bound this, we will require a technical lemma. The proof is a standard argument using the Taylor series of $\log(1+x)$. We include it for completeness. 

\begin{lem}\label{qkm}
Let $p_1(k)$ and $p_2(m)$ be polynomials such that $\deg p_1(k) \leq \deg p_2(k)-2$.
Then 
\[ \lim_{m \to \infty}\prod_{k=1}^{m} \left( 1+ \frac{a}{m} + \frac{bk}{m^2} + \frac{p_1(k)}{p_2(m)} \right) 
= e^{a+\frac{b}{2}} \]
\end{lem}

\begin{proof}
Let 
\[
 L = \log \left(\lim_{m \to \infty}\prod_{k=1}^{m} \left( 1+ \frac{a}{m} + \frac{bk}{m^2} + \frac{p_1(k)}{p_2(m)} \right)\right) 
,\] 
and we will show that $L = a + \frac{b}{2}$.  We have 
 
 \[ L = \lim_{m \to \infty} \sum_{k=1}^m \log \left(1+ \frac{a}{m} + \frac{bk}{m^2} + \frac{p_1(k)}{p_2(m)} \right) 
\]
 
We can use the Taylor series expansion for $\log(1+x)$ to get
\begin{align*} 
L &= \lim_{m \to \infty} \sum_{k=1}^{m}  \sum_{i=1}^{\infty} \frac{(-1)^{i+1}}{i} \left( \frac{a}{m} + \frac{bk}{m^2} + \frac{p_1(k)}{p_2(m)} \right)^i \\
&= \lim_{m \to \infty} \sum_{k=1}^{m} \left( \frac am + \frac{bk}{m^2} + \frac{p_1(k)}{p_2(m)} + \sum_{i=2}^{\infty} \frac{(-1)^{i+1}}{i} \left( \frac{a}{m} + \frac{bk}{m^2} + \frac{p_1(k)}{p_2(m)} \right)^i \right) \\
&= \lim_{m \to \infty} \left( \sum_{k=1}^{m} \frac am + \frac{bk}{m^2} + \frac{p_1(k)}{p_2(m)}\right) + \left(\sum_{k=1}^{m} \sum_{i=2}^{\infty} \frac{(-1)^{i+1}}{i} \left( \frac{a}{m} + \frac{bk}{m^2} + \frac{p_1(k)}{p_2(m)} \right)^i \right)
\end{align*}
We will evaluate the two terms above separately.

For the first term, we have
\[
\sum_{k=1}^{m} \frac am + \frac{bk}{m^2} + \frac{p_1(k)}{p_2(m)} 
= a + \frac{b}{m^2} \frac{m(m+1)}{2} + \frac{q(m)}{p_2(m)}, \]
where $q(m)$ is a polynomial of degree $\deg p_1(k) + 1 < \deg p_2(m)$. Thus the limit of this term is $a + \frac b2$. It remains to show that the other term above goes to 0 in the limit.

Since the Taylor series converges uniformly we may interchange limits as follows:
\[ \lim_{m \to \infty} \sum_{k=1}^{m} \sum_{i=2}^{\infty} \frac{(-1)^{i+1}}{i} \left( \frac{a}{m} + \frac{bk}{m^2} + \frac{p_1(k)}{p_2(m)} \right)^i 
= \sum_{i=2}^\infty \frac{(-1)^i}i \lim_{m\to \infty} \sum_{k=1}^m \left( \frac{a}{m} + \frac{bk}{m^2} + \frac{p_1(k)}{p_2(m)} \right)^i \]
 
Since $i\ge 2$ above, each term in the inner sum expands to a sum of terms in which the degree in $m$ of the denominator is at least 2 more than the degree in $k$ of the numerator. The summation in $m$ will produce terms in which the degree in $m$ of the denominator is at least 1 more than the degree in $m$ of the numerator, and thus the limit in $m$ of the inner sum will be zero. Thus we have
\[ \lim_{m \to \infty} \sum_{k=1}^{m} \sum_{i=2}^{\infty} \frac{(-1)^{i+1}}{i} \left( \frac{a}{m} + \frac{bk}{m^2} + \frac{p_1(k)}{p_2(m)} \right)^i = 0 \]
as desired.
\end{proof}

We are interested in bounding $L = \lim_{n \to \infty} D(V_n)$. We have shown, by Theorem \ref{upperbd} and Lemma \ref{qkm}, that 
\[ 0 \le L \leq \lim_{n \to \infty}\left(1-\frac{1}{n} + \frac{1}{2n(2n-1)}\right)^{n}= e^{-1} \approx .3678 \]
This asymptotic bound of $\frac{1}{e}$ matches the experimental data so closely that we make the following conjecture:

\begin{conj}
\[ \lim_{n\to \infty} D(V_n) = \frac{1}{e} \]
\end{conj}

We will not be able to prove the full conjecture, but
in the next section we will improve our lower bound to show
 that $ \frac{1}{e^3} \leq L $, and give very strong evidence that $\frac{1}{e^2} \le L$.

\section{Improving our lower bound}
\label{new}
In this section we will give a stronger version of the argument used in Lemma \ref{dirrel} and Theorem \ref{derthm}, allowing the letters $s_i$ and $l_i$ to become words of arbitrary length. 

 For a word $w$, we will denote the initial subword of length $k$ as $w|_k$ and the inverse of the terminal subword of length $k$ as $w|^k$.  If neither $a_k$ nor $a_k^{-1}$ appears in a word $w$, we will say that  $w$ is $a_k$-free. 

\begin{defn}\label{vnontrivial}
Let $\phi: \langle a_1, \dots, a_n \rangle \to  \langle a_1, \dots, a_n \rangle$ be given.
If, when we construct Wagner tails for $\phi$, the words $v$ and $\bar v$ are all nontrivial (except at the base point), then we say that $\phi$ is $v$-nontrivial. This is equivalent to the condition that for all $k$,
we can write
$ \phi(a_k)=s_ka_k^{\pm1}m_ka_k^{\pm1}l_k^{-1}, $ where $s_k$ and $l_k$ are nontrivial $a_k$-free words. 
 If $\phi$ is $v$-nontrivial, let $x_k = \left| s_k \right|$ and $y_k = \left| l_k \right|$ for each $k$. 
Note that $s_k= \phi(a_k)|_{x_k}$ and $l_k=\phi(a_k)|^{y_k}$. 
\end{defn}

The following lemma and its proof are a generalization of Lemma \ref{dirrel}, where we had $x_k=y_k=1$ for all $k$. 

\begin{lem}
\label{nicesowecken}
Let $X$ be a space such that $\pi_1(X)= \langle a_1, \dots, a_n \rangle$.  Let $f: X \to X$ be a map whose induced homomorphism $\phi: \pi_1(X) \to \pi_1(X)$  is $v$-nontrivial. 

If for all $k= 1, 2, \dots, n$ we have $l_k = \phi(a_k)|^{y_k} \neq \phi(a_k)|_{y_k}$ and additionally that $s_k = \phi(a_k)|_{x_k}$ does not equal $\phi(a_i)|_{x_k}$ or $\phi(a_i)|^{x_k}$ for any $i<k$ and
$l_k = \phi(a_k)|^{y_k}$ does not equal $\phi(a_i)|_{y_k}$ or $\phi(a_i)|^{y_k}$ for any $i<k$, then there are no direct relations among the Wagner tails of $\phi$, so $f$ is Wecken. 

\end{lem}

\begin{proof}

Let $x$ be a fixed point of $f$ identified with an occurrence of $a_i^{\pm1}$ in $\phi(a_i)$ with associated Wagner tails $w_x, \overline{w}_x$. The words $s_i$ and $l_i$ are non-trivial, so $w_x$ (which begins with $s_i$) and $\overline{w}_x$, (which begins with $l_i$) are non-trivial. So $x$ is not directly related to $x_0$. 

Let $y \neq x $ be a fixed point of $f$ identified with an occurrence of $a_j^{\pm1}$ in $\phi(a_j)$, with Wagner tails $w_y, \overline{w}_y$. 

Case 1, $i=j$:
Assume without loss of generality that $x$ comes before $y$ in the word $\phi(a_{i})$. 
Then $\left| w_{x} \right| < \left| w_{y} \right|$ and $\left| \overline{w}_{x} \right| > \left| \overline{w}_{y} \right|$. So $w_x\neq w_y$ and $\overline{w}_x \neq \overline{w}_y$.  
Assume for the sake of a contradiction that $w_x = \overline{w}_y$. 
Then, in particular, $\phi(a_i)|^{y_i}$
= $\overline{w}_y|_{y_i}= w_x|_{y_i}=$
  $\phi(a_i)|_{y_i}$. But by assumption $\phi(a_i)|^{y_i} \neq \phi(a_i)|_{y_i}$, which is a contradiction. 
We can similarly show that $\overline{w}_x \neq w_y$. 

Case 2, $i \neq j$:
Assume without loss of generality that $i<j$
and assume for the sake of a contradiction that $w_x = w_y$.
Then $ \left| w_x \right| = \left|w_y\right| \geq x_j.$ So $w_x$ contains at least $x_j$ letters, and so $w_x|_{x_j}$ is defined and equals $\phi(a_i)|_{x_j}$. But since $w_x=w_y$, we have that $w_x|_{x_j}= w_y|_{x_j}= \phi(a_j)|_{x_j}$.  But since $i<j$ we have that $\phi(a_j)|_{x_j} \neq \phi(a_i)|_{x_j}$. A similar argument shows that $w_x \neq \overline{w}_y$, $\overline{w}_x \neq w_y$, and $\overline{w}_x \neq \overline{w}_y$. 

So $\{w_x, \overline{w}_x\} \cap \{w_y, \overline{w}_y\} = \emptyset$ and $x$ is not directly related to $y$. 
\end{proof}

Let $K_n \subset W_n$ be the set of homomorphisms satisfying the conditions of Lemma \ref{nicesowecken}. So $D(K_n) \leq D(W_n)$. In order to calculate $D(K_n)$, we obtain a lower bound for $\left| G_p^n \cap K_n\right|$.

\begin{lem}
\label{xkyklem}
$$
\left| G_p^n \cap K_n \right| \geq  \prod_{k=1}^{n} \sum_{x=1}^{p-4} \sum_{y=1}^{p-x-3}   4X_kZ_kY_k
$$

where 
\begin{align*}
X_k&= (2n-2)(2n-3)^{x-1}-2(k-1)\\
Y_k&= (2n-2)(2n-3)^{y-1}-2(k-1)-1 \\
Z_k&= (2n-1)^{p-x-y-2}-1. 
\end{align*}
\end{lem}

\begin{proof}

We will count the number of ways to construct a homomorphism $\phi \in K_n \cap G_p^n$. Let $s_k, m_k,l_k, x_k,$ and $y_k$ be as in Definition \ref{vnontrivial}. 
Since $\phi \in K_n$, we have that $x_k$ and $ y_k$ are at least $1$ and $x_k+y_k \leq p-3$, and so  $\left| m_k \right| = p-x_k-y_k-2 \geq 1$. 
Then we have all of $\left| s_k \right|, \left| m_k \right|, $ and $\left| l_k^{-1} \right|$ greater than or equal to 1. 

For given $x=x_k$ and $y=y_k$, we want to find the number of ways to construct $s_k, m_k,$ and $l_k$ such that the conditions of Lemma \ref{nicesowecken} are satisfied at each $k$. 

For $k\geq1$, assume that we have already determined $\phi(a_1), \phi(a_2), \dots, \phi(a_{k-1})$. 
The word $s_k$ can be any $a_k$-free word of length $x$ that does not equal $\phi(a_i)|_{x}$ or $\phi(a_i)|^{x}$ for any $i<k$. So there are at least $(2n-2)(2n-3)^{x-1} - 2(k-1)=X_k$ choices. 

Similarly, the word $l_k$ can be any $a_k$-free word of length $y$ that does not equal $\phi(a_i)|_{y}$ or $\phi(a_i)|^{y}$ for any $i<k$ and that does not equal $\phi(a_k)|_{y}$. So there are at least $(2n-2)(2n-3)^{y-1}- 2(k-1)-1=Y_k$ ways to determine $l_k$. Note that $Y_k$ can possibly be zero or negative. In practice this occurs very infrequently for $n>2$ and will not significantly affect our count. For the $n=2$ case we will see that the lower bound given in this Lemma is not useful.

If we assume that  $|\phi(a_k)|=p_k$ then there are at least $(2n-1)^{p_k-x-y-2}(2n-2)$ choices for the word $m_k$, whose first and last letters are restricted. Summing over $5 \leq p_k \leq p$, we get at least $(2n-1)^{p-x-y-2}-1= Z_k$ ways to determine $m_k$. 

Finally, for each of the two required $a_k^{\pm1}$, we can choose whether the exponent is $1$ or $-1$ in one of 4 ways. So for each selection of $x$ and  $y$ there are at least 
$ 4X_kZ_kY_k$ choices for $\phi(a_k)$. 
So, if we let $x$ and $y$ range from $1$ to $p-3$, there are at least
\[  \sum_{x=1}^{p-4} \sum_{y=1}^{p-x-3} 4X_kZ_kY_k\] ways to determine $\phi(a_k)$, and 

$$  \prod_{k=1}^{n}\sum_{x_k=1}^{p-4} \sum_{y_k=1}^{p-x-3}  4X_kZ_kY_k$$ ways to determine $\phi$. 
\end{proof}

\begin{thm}
\label{xkykdens}
 Let 
\[ d_n= \prod_{k=1}^{n}
 \frac{ 4(n-1)^4-(8k-6)(n-1)^2+4k^2-6k+2}{n(n-1)(2n-1)^2} 
\]
Then if $D(V_n)$ exists we have $D(V_n) \geq d_n$. 

\end{thm}

\begin{proof}
By Lemma \ref{xkyklem} we have that 
\[
|G_p^n \cap V_n|\geq|G_p^n \cap K_n|\geq  \prod_{k=1}^{n}\sum_{x=1}^{p-4} \sum_{y=1}^{p-x-3}  4X_kZ_kY_k, \]
and so
\[D(V_n) \geq \lim_{p \to \infty} \frac{ |G_p^n \cap K_n|}{|G_p|^n} \geq \lim_{p \to \infty} \prod_{k=1}^{n} \frac{1}{|G_p|} \sum_{x=1}^{p-4} \sum_{y=1}^{p-x-3}  4X_kZ_k Y_k. \]

The above sums involve $x$ and $y$ only in exponents with bases $2n-1, 2n-2$, and $2n-3$. Thus the sums can be evaluated as finite geometric series. This can be done by hand with some effort, a task made easier by the fact that we can ignore terms that go to zero in the limit as $p\to \infty$. Alternatively this can be done by computer (Sage has trouble simplifying the sums, but Mathematica evaluates them without problems) and we obtain
\[
\lim_{p \to \infty} \frac{1}{|G_p|} \sum_{x=1}^{p-4} \sum_{y=1}^{p-x-3}  4X_kZ_k Y_k = \frac{4(n-1)^4-(8k-6)(n-1)^2+4k^2-6k+2}{n(n-1)(2n-1)^2}\\
\]
which gives the desired bound.
\end{proof}

This lower bound takes the following values: 

\[ \begin{tabular}{ c|  c c c c c c }
$n$ & 2& 3&5 & 10&  20 & 50  \\
\hline
$d_n$ & 0 & .0059 & .0209 & .0348 & .0421 & .0467
\end{tabular} \]
(The value $d_2$ is exactly 0.)
The values of $d_n$ seem to approach a nonzero limit in $n$, and in fact, we can compute this limit exactly.

\begin{lem}
\label{denslim1}
\[ \lim_{n \to \infty}d_n= \frac{1}{e^3} \approx .0497\]
\end{lem}

\begin{proof}
We have
\[
\lim_{n \to \infty} d_n= \lim_{n \to \infty} \prod_{k=1}^{n}
 \frac{ 4(n-1)^4-(8k-6)(n-1)^2+4k^2-6k+2}{n(n-1)(2n-1)^2} 
  \]

Using polynomial long division, we have  
\[
\lim_{n \to \infty}D(K_n) = \lim_{n \to \infty} \prod_{k=1}^{n} \left(
1 - \frac 2 n - \frac{2k}{n^2} + Q(k,n) \right)
\]
where $Q(k,n)$ is a ratio of polynomials of the form appearing in Lemma \ref{qkm} and so by that lemma the above limit is $e^{-2-2/2} = e^{-3}$. 
\end{proof}

In our construction of $K_n$, we have unnecessarily required that $\phi(a_i)$ neither begins nor ends with $a_i^{-1}$. We can improve (though complicate) our lower bound by allowing this to occur. 

We say that a homomorphism $\phi$ is $w$-nontrivial when all of the Wagner tails (except those of the base point) are nontrivial. This class of homomorphisms includes the $v$-nontrivial homomorphisms discussed above. 

When we examine the image words $\phi(a_k)$ for a $w$-nontrivial homomorphism, they will come in one of the following types:
If $\phi(a_k)$ can be written in reduced form as $s_ka_k^{\pm1}m_ka_k^{\pm1}l_k^{-1}$ where $s_k$ and $l_k$ are $a_k$-free, we will say it is \emph{Type 0}. If it can be written as $a_k^{-1}m_k a_k^{\pm1}l_k^{-1}$ it is  \emph{Type 1a}, and in this case we define $s_k = a_k^{-1}$. If it can be written as $s_ka_k^{\pm1}m_ka_k^{-1}$, it is \emph{Type 1b}, and we define $l_k = a_k$. 
Finally if it can be written as $a_k^{-1}m_ka_k^{-1}$ it is \emph{Type 2}, and we define $s_k=a_k^{-1}$ and $l_k=a_k$. In all of these cases we define $x_k=|s_k|$ and $y_k=|l_k|$.
 
Note that if $\phi(a_k)$ is of Type 0 for all $k$, then $\phi$ is $v$-nontrivial.

\begin{lem}
\label{niceenough}
Let $f: X \to X$ be a map whose induced homomorphism $\phi: \pi_1(X) \to \pi_1(X)$ is $w$-nontrivial, and let $s_k, l_k$ be given as above.
Specify a particular ordering of the set
$$ \{ \phi(a_1), \phi(a_1)^{-1}, \dots, \phi(a_n), \phi(a_n)^{-1} \} \text{ as } \{v_1, \dots, v_{2n} \}, $$ And say that some particular $v_i$ is \emph{positive} or \emph{negative} according to the exponent in $v_i = \phi(a_i)^{\pm 1}$. 

If for all $j<i$ we have $v_i \neq v_j|_{|s_i|}$ when $v_i$ is positive, and $v_i \neq v_j|_{|l_i|}$ when $v_i$ is negative,
then $\phi \in V_n$. 
\end{lem}
\begin{proof}
Similar to that of Lemma \ref{nicesowecken}. In Lemma \ref{nicesowecken} we used the ordering where $v_1= \phi(a_1), v_2 = \phi(a_1)^{-1}, \dots, v_{2n-1}= \phi(a_n), v_{2n}= \phi(a_n)^{-1}$. The choice of this particular ordering is arbitrary, however, and so essentially the same proof will suffice.
\end{proof}
 
Let $L_n$ be the set of homomorphisms satisfying the conditions of the above Lemma. Then $ K_n \subset L_n \subseteq V_n$, and so we will be able to use $D(L_n)$ as an improved lower bound on $D(V_n)$. 

\begin{lem}

\[
| G_p^n \cap L_n | \geq \sum_{c=0}^{n} \sum_{b=0}^{n-c} \binom{n}{c} \binom{n-c}{b} 2^b R(n,p,c,b)
\]
where
\begin{align*}
 R(n,p,c,b)&=
S(n,p)^c
\left(
\prod_{j=1}^{b} \sum_{y=1}^{p-3} T(n,p,c,j, y)
\right) \left(
\prod_{k=c+b+1}^{n}
 \sum_{x=1}^{p-4}
  \sum_{y=1}^{p-x-3}4X_kZ_kY_k
  \right)
\\
S(n,p)&= (2n-1)^{p-2}-1,
\\
T(n,p,c,b, j, y)&=2 \left( (2n-1)^{p-y-2}-1 \right) \left( (2n-2)(2n-3)^{y-1}-(2c+b+j-1) \right)
\end{align*}
and $ X_k, Y_k, Z_k$ are defined as in Lemma \ref{xkyklem}.
\end{lem}

\begin{proof}
We will count the number of ways to construct a homomorphism $\phi \in L_n \cap G_p^n$.
Let $c$ be the number of $i$ such that  $\phi(a_i)$ is of Type 2. Then there are $\binom{n}{c}$ ways to select the $\phi(a_i)$ of Type 2. Let $b$ be the number of $i$ such that $\phi(a_i)$ is of Type 1. Then there are $\binom{n-c}{b}$ ways to choose which of the remaining $\phi(a_i)$ will be Type 1, and $2^b$ ways to determine whether each $\phi(a_i)$ of Type 1 is of Type 1a or Type 1b. 

Write $a_1, \dots, a_n$ according to the type of $\phi(a_i)$, in descending order, as 
\[ a_{i_1}, \dots, a_{i_c}, a_{j_1}, \dots, a_{j_b}, a_{k_1}, \dots a_{k_{n-b-c}}, \]
so that each $\phi(a_{i_*})$ is Type 2, each $\phi(a_{j_*})$ is Type 1, and each $\phi(a_{k_*})$ is Type 0.
 We will assume without loss of generality that all $\phi(a_j)$ of Type 1 are of Type 1a. (Distinguishing between 1a and 1b will not affect the counts.)
Then we will use the following ordering of $s_1, l_1, \dots, s_n, l_n$:
\begin{align*}
\{v_1, \dots, v_{2n}\} = \{ &\phi(a_{i_1}), \phi(a_{i_1})^{-1}, \dots, \phi(a_{i_c}), \phi(a_{i_c})^{-1}, \\
&\phi(a_{j_1}), \dots, \phi(a_{j_b}), \phi(a_{j_1})^{-1}, \dots, \phi(a_{j_b})^{-1}, \\
&\phi(a_{k_1}), \phi(a_{k_1})^{-1}, \dots \phi(a_{k_{n-b-c}}),  \phi(a_{k_{n-b-c}})^{-1} \}
\end{align*}

We will now count the ways to choose the words $\phi(a_i)$ such that the ordering above satisfies the condition of Lemma \ref{niceenough}, and thus $\phi \in L_n$. 

For each $i$, let $p_i = |\phi(a_i)|$. 

We will first determine the $\phi(a_i)$ that are Type 2, that is, $\phi(a_i) = a_i^{-1}m_ia_i^{-1}$. For these words we have $|s_i| = |l_i| = 1$ and the first letters of $\phi(a_i)^{\pm 1}$ are all distinct. Thus the Lemma \ref{niceenough} will automatically be satisfied for the partial list 
\begin{equation}\label{partial1}
\{v_1, \dots, v_{2c} \} = 
\{ \phi(a_{i_1}), \phi(a_{i_1})^{-1}, \dots, \phi(a_{i_c}), \phi(a_{i_c})^{-1} \}.
\end{equation}
Thus we may choose the words $m_i$ freely without being careful to satisfy Lemma \ref{niceenough}.

There are at least $(2n-1)^{p_i-3}(2n-2)$ ways to determine $m_i$ and thus $\phi(a_i)$. Letting $p_i$ vary from $3$ to $p$, we have
\[ \sum_{p_i=3}^{p}(2n-1)^{p_i-3}(2n-2)= (2n-1)^{p-2}-1= S(n,p) \]
ways to determine each individual $\phi(a_i)$, and thus 
$\prod_{i=1}^{c}S(n,p) = S(n,p)^c$ ways to determine all of the $\phi(a_i)$ of Type 2. 
 
Now we will determine the $\phi(a_j)$ of Type 1, which we have assumed are Type 1a. Since these are of Type 1a we will have $s_{i_j} = a_{i_j}^{-1}$, and thus the partial list
\begin{equation}\label{partial2}
\{ v_1, \dots, v_{2c+b} \}
= \{ \phi(a_{i_1}), \phi(a_{i_1})^{-1}, \dots, \phi(a_{i_c}), \phi(a_{i_c})^{-1}, \phi(a_{j_1}), \dots, \phi(a_{j_b}) \}
\end{equation}
satisfies the conditions of Lemma \ref{niceenough}.

Since $\phi(a_j)$ is Type 1a, we have $\phi(a_j) = a_j^{-1}m_j a_j^{\pm1}l_j^{-1}$.
First choose the exponent of the $a_j^{\pm1}$ immediately preceding $l_j^{-1}$ in one of 2 ways. 
For each $j=1, 2, \dots b$, let $y_j$ denote the length of $l_j$ and we can freely choose $y_1, \dots, y_b$. 

For a given $p_j=|\phi(a_j)|$  there are $(2n-1)^{p_j-y_j-3}(2n-2)$ ways to determine $m_j$.
So there are $\sum_{p_j=3}^{p}(2n-1)^{p_j-y_j-3}(2n-2)= (2n-1)^{p-y_j-2}-1$ ways to determine $m_j$.

Now we have to determine $l_{j_h}$ for each $h=1, 2, \dots, b$ such that the conditions of Lemma \ref{niceenough} are fulfilled in the partial list
\begin{equation}\label{partial3}
\{ v_1, \dots, v_{2c+b} \} = 
\{ \phi(a_{i_1}), \phi(a_{i_1})^{-1}, \dots, \phi(a_{i_c}), \phi(a_{i_c})^{-1}, \phi(a_{j_1}), \dots, \phi(a_{j_b}), \phi(a_{j_1})^{-1}, \dots, \phi(a_{j_b})^{-1} \}. 
\end{equation}

For $y=y_{j_h} = |l_{j_h}|$ there are $(2n-2)(2n-3)^{y-1}$ ways to select $l_{j_h}$ to be $a_{j_h}^{\pm1}$-free. However, we cannot choose $l_{j_h}$ such that it equals the previously chosen initial subwords of length $y$ in the list \eqref{partial3}, of which there are at most $2c+b+h-1$.  So we have at least $(2n-2)(2n-3)^{y-1}- (2c+b+h-1)$ ways to determine $l_{j_h}$ so that the condition of Lemma \ref{niceenough} holds for the partial list \eqref{partial3}.

So, over our $\phi(a_{j_h})$ of Type 1, we have 
\[ \prod_{j=1}^{b}2( (2n-1)^{p-y-2}-1) ((2n-2)(2n-3)^{y-1}- (2c+b+j-1)) \] choices.
Summing over the possible values of $y$, we have
\begin{align*} 
\sum_{y=1}^{p-3} \prod_{j=1}^{b} &2 \left( (2n-1)^{p-y-2}-1\right)\left( (2n-2)(2n-3)^{y-1}- (2c+b+j-1)\right) \\
&= \prod_{j=1}^{b} \sum_{y=1}^{p-3} T(n,p,c,b,j, y)
\end{align*}
choices for the $\phi(a_{j_h})$ of Type 1.

Finally, we will consider our $\phi(a_{k_h})$ of Type 0, which can be written as $s_{k_h}a_{k_h}^{\pm1}m_{k_h}a_{k_h}^{\pm1}l_{k_h}^{-1}$.
The count for these words is exactly analogous to the count used in Lemma \ref{xkyklem}. The difference is that the $2(k-1)$ in the definition of $X_k$ must become $2(c+b+h-1)$ since there will be this many restrictions on the choice of the word $s_{k_h}$ at this stage in order for the condition of Lemma \ref{niceenough} to be satisfied. Similarly the $2(k-1) - 1$ in the definition of $Y_k$ must become $2(c+b+h-1)-1$. These changes are equivalent to replacing $k$ with $c+b+h$, and thus we have at least
\[ \sum_{x=1}^{p-4} \sum_{y=1}^{p-x-3} 4X_{c+b+h}Z_{c+b+h}Y_{c+b+h} \]
choices for $\phi(a_{k_h})$. 

There are $n-b-c$ such words of Type 0, so we have a total of 
\[ 
\prod_{h=1}^{n-c-b} \sum_{x=1}^{p-4} \sum_{y=1}^{p-x-3} 4X_{c+b+h}Z_{c+b+h}Y_{c+b+h} = \prod_{k=c+b+1}^n \sum_{x=1}^{p-4} \sum_{y=1}^{p-x-3} 4X_kZ_kY_k
\]
choices for all the words of Type 0.

So for a given $b$ and $c$ we have \[
\binom{n}{c} \binom{n-c}{b} 2^b 
S(n,p)^c
\left(
\prod_{j=1}^{b} \sum_{y=1}^{p-3} T(n,p,c,b,j, y)
\right) \left(
\prod_{k=c+b+1}^{n}
 \sum_{x=1}^{p-4}
  \sum_{y=1}^{p-x-3}4X_kZ_kY_k
  \right)
\]
ways to determine $\phi \in G_p^n \cap L_n$. Summing over $c$ and $b$ gives us the desired result. 
\end{proof}

\begin{thm}\label{bestbound}
\[
D(L_n) \geq d^*_n= \sum_{c=0}^{n} \sum_{b=0}^{n-c} \binom{n}{c} 
 \binom{n-c}{b} 2^b \left( \frac{n-1}{n(2n-1)^2} \right)^c
K
\prod_{j=1}^{b} \frac{2(n-1)^2- (2c+b+j-1)}{n(2n-1)^2}
\]
where
\[ K=  
 \prod_{k=c+b+1}^{n}
 \frac{4(n-1)^4-(8k-6)(n-1)^2 + 4k^2-6k+2}{n(n-1)(2n-1)^2}
\]

\end{thm}

\begin{proof}
By the previous lemma we have
\[
\lim_{p\to\infty} \frac{|L_n\cap G_p^n|}{|G_p|^n} 
\ge \sum_{c=0}^n\sum_{b=0}^{n-c} \binom{n}{c} \binom{n-c}b 2^b \lim_{p\to\infty} \frac{1}{|G_p|^n} R(n,p,c,b),
\]
and 
\begin{equation}\label{rprod}
\frac{1}{|G_p|^n} R(n,p,c,b) \ge \left(  \frac{S(n,p)}{|G_p|} \right)^c \left( \prod_{j=1}^b \sum_{y=1}^{p-c} \frac{T(n,p,c,b,j,y)}{|G_p|} \right) \left( \prod_{k=c+b+1}^n \frac{X_kZ_kY_k)}{|G_p|} \right).
\end{equation}
We will estimate each of the three factors on the right in the limit as $p\to\infty$.

The first factor of \eqref{rprod} is easily evaluated. The formulas for $S(n,p)$ and $|G_p|$ give:
\[ \lim_{p\to\infty} \frac{S(n,p)}{|G_p|} = \frac{n-1}{n(2n-1)^2} \]

The terms of the third factor of \eqref{rprod} are exactly the terms encountered in Theorem \ref{xkykdens}. By the same argument, a complicated simplification preferably done by computer, we have 
\[ \lim_{p\to\infty} \prod_{k=c+b+1}^n \frac{X_kZ_kY_k}{|G_p|} = K. \]

For the second factor of \eqref{rprod} we require another complicated summation most easily done on a computer. Again Mathematica computes the limit as follows:
\[ \lim_{p\to\infty} \prod_{j=1}^b \sum_{y=1}^{p-c} \frac{T(n,p,c,b,j,y)}{|G_p|} =  \prod_{j=1}^{b} \frac{2(n-1)^2- (2c+b+j-1)}{n(2n-1)^2}\]

Combining the three above calculations in \eqref{rprod} gives the result.
\end{proof}

This lower bound takes the following values:
\[
\begin{tabular}{c|cccccc}
$n$ & 2 & 3 & 5 & 10 & 20 & 50 \\
\hline
$d_n^*$ & $<0$ & .0252 & .0694 & .1029 & .1193 & .1289
\end{tabular}
\]
These values appear to approach a limit as $n\to \infty$, but the formula for $d_n^*$ is too complicated to easily evaluate the limit. Nevertheless the following conjecture seems clear:
\begin{conj}
$\displaystyle \lim_{n \to \infty} d^*_n = e^{-2} \approx .1353$
\end{conj}

The above computed values for $d^*_n$ are not obviously tending to $e^{-2}$, but some higher values of $n$ make the limit a bit clearer:
\[
\begin{tabular}{c|cccccc}
$n$ & 500 & 1000 & 1500 & 2000 & 2500 \\
\hline
$d_n^*$ & .1347 &  .1350 & .1351 & .1351 & .1352
\end{tabular}
\]

\end{document}